\documentclass[12pt]{amsart}
\usepackage{amsmath}
\usepackage{amsthm}
\usepackage{color}
\usepackage{amscd}
\usepackage{amsfonts}
\usepackage{graphicx}
\usepackage{epsfig}
\usepackage{amssymb,latexsym}
\numberwithin{equation}{section}

\newcommand{\dv}{\operatorname{div}}

\def\bu{\mathbf{u}}

\def\bv{\mathbf{v}}
\def\bw{\mathbf{w}}
\def\bW{\mathbf{W}}
\def\bU{\mathbf{U}}
\def\bV{\mathbf{V}}
\def\bx{\mathbf{x}}
\def\bz{\mathbf{z}}
\def\by{\mathbf{y}}
\def\ba{\mathbf{a}}

\def\bbf{\mathbf{f}}
\def\R{\mathbb{R}}
\def\bZ{\mathbf{Z}}
\def\bbD{\mathbb{D}}

\def\bbH{\mathbf H}
\def\bbH{\mathbb H}

\def\T{\mathbf T}

\def\bA{\mathbf A}

\def\C{\mathcal C}
\def\bbS{\mathbb S}
\def\cS{\mathcal S}

\def\bcero{\mathbf 0}

\newtheorem{theorem}{Theorem}[section]

\newtheorem{definition}{Definition}[section]
\newtheorem{lemma}[theorem]{Lemma} 
\newtheorem{proposition}[theorem]{Proposition}

\begin{document}

%\title[Unifying viewpoint]{On a unifying, variational standpoint for Differential Equations}

\title[Uniqueness]{On a general variational framework for existence and uniqueness in Differential Equations}

%\titlerunning{Recovering algorithm}
\author{Pablo Pedregal}
\date{} % delete this line to display the current date
\thanks{Supported by the Spanish Ministerio de Ciencia, Innovación y Universidades through project MTM2017-83740-P}
\begin{abstract}
Starting from the classic contraction mapping principle, we establish a general, flexible, variational setting that turns out to be applicable to many situations of existence in Differential Equations. We show its potentiality with some selected examples including initial-value, Cauchy problems for ODEs; non-linear, monotone PDEs; linear and non-linear hyperbolic problems; and steady Navier-Stokes systems. 
\end{abstract}
%\end{abstract}
%%% BEGIN DOCUMENT
%\shorttitle{Recovering algorithm}

\maketitle
%\tableofcontents
\section{Introduction}
Possibly the most fundamental result yielding existence and uniqueness of solutions of an equation is the classic Banach contraction mapping principle.
\begin{theorem}\label{contraction}
Let $\T:\bbH\to\bbH$ be a mapping from a Banach space $\bbH$ into itself that is contractive in the sense
$$
\|\T\bx-\T\by\|\le k\|\bx-\by\|,\quad k\in[0, 1), \bx, \by\in\bbH.
$$
Then $\T$ admits a unique fixed point $\overline\bx\in\bbH$,
$$
\T\overline\bx=\overline\bx.
$$
\end{theorem}
The proof is well-known, elementary, and independent of dimension. The most fascinating issue is that this basic principle is at the heart of many uniqueness results in 
Applied Analysis and Differential Equations. Our aim is to stress this fact from a variational stand-point. This means that we would like to rephrase the previous principle into a variational form that could be directly and flexibly used in many of the situations where uniqueness of solutions is known or expected. Our basic principle is the following.
\begin{proposition}\label{basica}
Let $E:\bbH\to\R^+$ be a non-negative, lower semi-continuous functional in a Banach space $\bbH$, such that
\begin{equation}\label{enhcoerf}
\|\bx-\by\|\le C(E(\bx)+E(\by)),\quad C>0, \bx, \by\in\bbH.
\end{equation}
Suppose, in addition, that
\begin{equation}\label{infcero}
\inf_{\bz\in\bbH}E(\bz)=0.
\end{equation}
Then there is a unique minimizer, i.e. a unique $\overline\bx\in\bbH$ such that $E(\overline\bx)=0$, and
\begin{equation}\label{errorint}
\|\bx-\overline\bx\|\le CE(\bx),\quad \bx\in\bbH.
\end{equation}
\end{proposition}
The proof again is elementary, because every minimizing sequence $\{\bx_j\}$ with $E(\bx_j)\searrow0$ must be a Cauchy sequence in $\bbH$, according to \eqref{enhcoerf}, and so it converges to some $\overline\bx\in\bbH$. The lower semicontinuity implies that
$$
0\le E(\overline\bx)\le\liminf_{j\to\infty}E(\bx_j)=0,
$$
and $\overline\bx$ is a minimizer. Condition \eqref{enhcoerf} implies automatically that such minimizer is unique, and leads to \eqref{errorint}. 

Condition \eqref{errorint} is a very clear statement that functional $E$ in Proposition \ref{basica} is a measure of how far we are from $\overline\bx$, the unique point where $E$ vanishes. Indeed, this consequence already points in the direction in which to look for functionals $E$ in specific situations: they should be setup as a way to measure departure from solutions sought. This will be taken as a guiding principle in concrete examples. The usual least-square method (see \cite{bogunz}, \cite{glowinski}, for example), suitably adapted to each situation, stands as a main, natural possibility for $E$. 

It is not surprising  that Proposition \ref{basica} is more general than Theorem \ref{contraction}, in the sense that the latter is a consequence of the former by considering the natural functional
\begin{equation}\label{error}
E(\bx)=\|\T\bx-\bx\|.
\end{equation}
Indeed, for an arbitrary pair $\bx, \by\in\bbH$, 
$$
\|\bx-\by\|\le\|\bx-\T\bx\|+\|\T\bx-\T\by\|+\|\T\by-\by\|,
$$
and
$$
\|\bx-\by\|\le E(\bx)+E(\by)+k\|\bx-\by\|.
$$
From here, we immediately find \eqref{enhcoerf}
$$
\|\bx-\by\|\le \frac1{1-k}(E(\bx)+E(\by)).
$$
Along every sequence of iterates, we have \eqref{infcero} if $\T$ is contactive. Of course, minimizers for $E$ in \eqref{error} are exactly fixed points for $\T$. 

Our objective is to argue that the basic variational principle in Proposition \ref{basica} is quite flexible, and can be implemented in many of the situations in Differential Equations where uniqueness of solutions is known. 

There are two main requisites in Proposition \ref{basica}. The first one \eqref{enhcoerf} has to be shown directly in each particular scenario where uniqueness is sought. Note that it is some kind of enhanced coercivity, and, as such, stronger than plain coercivity. 
Concerning \eqref{infcero}, there is, however, a general strategy based on smoothness that can be applied to most of the interesting situations in practice. For the sake of simplicity, we restrict attention to a Hilbert space situation, and regard $\bbH$ as a Hilbert space henceforth. 
If a non-negative functional $E:\bbH\to\R^+$ is $\C^1$-, then 
$$
\inf_{\bx\in\bbH}\|E'(\bx)\|=0.
$$
Therefore, it suffices to demand that
$$
\lim_{E'(\bx)\to\bcero}E(\bx)=0
$$
to enforce \eqref{infcero}. Proposition \ref{basica} becomes then:
\begin{proposition}\label{basicad}
Let $E:\bbH\to\R^+$ be a non-negative, $\C^1$- functional in a Hilbert space $\bbH$, such that
\begin{equation}\label{enhcoers}
\|\bx-\by\|\le C(E(\bx)+E(\by)),\quad C>0, \bx, \by\in\bbH.
\end{equation}
Suppose, in addition, that
\begin{equation}\label{infceros}
\lim_{E'(\bx)\to\bcero}E(\bx)=0.
\end{equation}
Then there is a unique $\overline\bx\in\bbH$ such that $E(\overline\bx)=0$, and
$$
\|\bx-\overline\bx\|\le CE(\bx)
$$
for every $\bx\in\bbH$. 
\end{proposition}
Though the following is a simple observation, it is worth to note it explicitly. 
\begin{proposition}
Under the same conditions as in Proposition \ref{basicad}, the functional $E$ enjoys the Palais-Smale condition. 
\end{proposition}
We remind readers that the fundamental Palais-Smale condition reads:
\begin{quote}
If the sequence $\{\bx_j\}$ is bounded in $\bbH$, and $E'(\bx_j)\to\bcero$ in $\bbH$, then, at least for some subsequence, $\{\bx_j\}$ converges in $\bbH$.
\end{quote}
Again, it is not difficult to suspect the proof. Condition \eqref{infceros} informs us that Palais-Smale sequences ($\{\bx_j\}$, bounded and $E'(\bx_j)\to\bcero$) are always minimizing sequences for $E$ ($E(\bx_j)\to0$), while the estimate \eqref{enhcoers} ensures that (the full) such sequence is a Cauchy sequence in $\bbH$. Notice, however, that, due to \eqref{infceros}, $0$ is the only possible critical value of $E$, and so critical points become automatically global minimizers regardless of convexity considerations.

In view of the relevance of conditions \eqref{enhcoers} and \eqref{infceros}, we adopt the following definition in which we introduce some simple, helpful changes to broaden its applicability. We also change the notation to stress that vectors in $\bbH$ will be functions for us. 

\begin{definition}\label{errorgeneral}
A non-negative, $\C^1$-functional 
$$
E(\bu):\bbH\to\R^+
$$ 
defined over a Hilbert space $\bbH$ is called an error functional if 
\begin{enumerate}
\item behavior as $E'\to\bcero$:
\begin{equation}\label{comporinff}
\lim_{E'(\bu)\to\bcero}E(\bu)=0
\end{equation}
over bounded subsets of $\bbH$; and
\item enhanced coercivity: there is a positive constant $C$, such that for every pair $\bu, \bv\in\bbH$ 
we have
\begin{equation}\label{enhcoer}
\|\bu-\bv\|^2\le C(E(\bu)+E(\bv)).
\end{equation}
\end{enumerate}
\end{definition}
Our basic result Proposition \ref{basicad} remains the same.
\begin{proposition}\label{basicas}
Let $E:\bbH\to\R^+$ be an error functional according to Definition \ref{errorgeneral}. Then there is a unique minimizer $\bu_\infty\in\bbH$ such that  $E(\bu_\infty)=0$, and
\begin{equation}\label{medidaerror}
\|\bu-\bu_\infty\|^2\le CE(\bu),
\end{equation}
for every $\bu\in\bbH$.
\end{proposition}

It is usually said that the contraction mapping principle Theorem \ref{contraction}, though quite helpful in ODEs, is almost inoperative for PDEs. We will try to make an attempt at convincing readers that, on the contrary, Proposition \ref{basicad} is equally helpful for ODEs and PDEs. To this end, we will examine several selected examples as a sample of the potentiality of these ideas. Specifically, we will look at the following situations, though none of our existence results is new at this stage:
\begin{enumerate}
\item Cauchy, initial-value problems for ODEs;
\item linear hyperbolic examples;
\item non-linear, monotone PDEs;
\item non-linear wave models;
\item steady Navier-Stokes system.
\end{enumerate}
We systematically will have to check the two basic properties \eqref{enhcoer} and \eqref{comporinff} in each situation treated. 
We can be dispensed with condition \eqref{comporinff}, and replace it by \eqref{infcero} if more general results not requiring smoothness are sought. On the other hand, in many regular situations linearization may lead in a systematic way to the following.
\begin{proposition}\label{principalbis}
Let 
$$
E(\bu):\bbH\to\R^+
$$ 
be a $\C^1$-functional verifying the enhanced coercivity condition \eqref{enhcoer}. 
Suppose there is 
$\T:\bbH\to\bbH$, 
a locally Lipschitz operator, such that
\begin{equation}\label{propoper}
\langle E'(\bu), \T\bu\rangle=-dE(\bu)
\end{equation}
for every $\bu\in\bbH$, and some constant $d>0$. Then $E$ is an error functional (according to Definition \ref{errorgeneral}), and, consequently, there is a unique global minimizer $\bu_\infty$ with $E(\bu_\infty)=0$, and \eqref{medidaerror} holds
\begin{equation}\label{medidaerrorbis}
\|\bu-\bu_\infty\|^2\le CE(\bu),
\end{equation}
for some constant $C$, and every $\bu\in\bbH$. 
\end{proposition}
Note how condition \eqref{propoper} leads immediately to \eqref{comporinff}. 
In this contribution, we will assume smoothness in all of our examples. 

Typically our Hilbert spaces $\bbH$ will be usual Sobolev spaces in different situations, so standard facts about these spaces will be taken for granted. In particular, the following Hilbert spaces will play a basic role for us in those various situations mentioned above
$$
H^1(0, T; \R^N),\quad H^1_0(\R^N_+),\quad H^1(\R^N_+),\quad H^1_0(\Omega), \quad H^1_0(\Omega; \R^N),
$$
for a domain $\Omega\subset\R^N$ as regular as we may need it to be. 

If one is interested in numerical or practical approximation of solutions $\bu_\infty$, note how \eqref{medidaerrorbis} is a clear invitation to seek approximations to $\bu_\infty$ by minimizing $E(\bu)$. The standard way to take a functional to its minimum value is to use a steepest descent algorithm or some suitable variant of it. It is true that such procedure is designed, in fact, to lead the derivative $E'(\bu)$ to zero; but precisely, condition \eqref{comporinff} is guaranteeing that in doing so we are also converging to $\bu_\infty$ always. We are not pursuing this direction here, though it has been implemented in some scenarios (\cite{munped2}, \cite{pedregal3}). 

Definition \ref{errorgeneral} is global. A local parallel concept may turn out necessary for some situations. We will show this in our final example dealing with the steady Navier-Stokes system. The application to parabolic problems, though still feasible, is, in general, more delicate. 

% ================================================
\section{Cauchy problems for ODEs}
As a preliminary step, we start testing our ideas with a typical  initial-value, Cauchy problem for the non-linear system
\begin{equation}\label{ode}
\bx'(t)=\bbf(\bx(t))\hbox{ in }(0, +\infty),\quad \bx(0)=\bx_0
\end{equation}
where the map 
$$
\bbf(\by):\R^N\to\R^N
$$ 
is smooth and globally Lipschitz, and $\bx_0\in\R^N$. Under these circumstances, it is well-known that \eqref{ode} possesses a unique solution.
Let us pretend not to know anything about problem \eqref{ode}, and see if our formalism could be applied in this initial situation to prove the following classical theorem.

\begin{theorem}
If the mapping $\bbf(\by)$ is globally Lipschitz, there is unique absolutely continuous solution 
$$
\bx(t):[0, \infty)\to\R^N
$$ 
for \eqref{ode}.
\end{theorem}

According to our previous discussion, we need a functional $E:\bbH\to\R^+$ defined on an appropriate Hilbert space $\bbH$ complying with the necessary properties. 

For a fixed, but otherwise arbitrary, positive time $T$, we will take 
\begin{gather}
\bbH=\{\bz(t):[0, T]\to\R^N: \bz\in H^1(0, T; \R^N), \bz(0)=\bcero\},\nonumber\\
E(\bz)=\frac12\int_0^T|\bz'(s)-\bbf(\bx_0+\bz(s))|^2\,ds\label{erroredo}.
\end{gather}
$\bbH$ is a subspace of the standard Sobolev space $H^1(0, T; \R^N)$, under the norm (recall that $\bz(0)=\bcero$ for paths in $\bbH$)
$$
\|\bz\|^2=\int_0^T|\bz'(s)|^2\,ds.
$$
Note that paths $\bx\in\bbH$ are absolutely continuous, and hence $E(\bz)$ is well-defined over $\bbH$. 
We first focus on \eqref{enhcoer}.
\begin{lemma}
For paths $\by$, $\bz$ in $\bbH$, we have
$$
\|\by-\bz\|^2\le C(E(\by)+E(\bz)),\quad C>0.
$$
\end{lemma}
\begin{proof}
The proof is, in fact, pretty elementary. Suppose that 
$$
\bz(0)=\by(0)=\bx_0,
$$ 
so that $\by-\bz\in\bbH$. Then
\begin{align}
\by(t)-\bz(t)=&\int_0^t (\by'(s)-\bz'(s))\,ds\nonumber\\
=&\int_0^t (\by'(s)-\bbf(\by(s)))\,ds+\int_0^t (\bbf(\by(s))-\bbf(\bz(s)))\,ds\nonumber\\
&+\int_0^t (\bbf(\bz(s))-\bz'(s))\,ds.\nonumber
\end{align}
From here, we immediately find
$$
|\by(t)-\bz(t)|^2\le C(E(\by)+E(\bz))+CM^2\int_0^t|\by(s)-\bz(s)|^2\,ds,
$$
if $M$ is the Lipschitz constant for the map $\bbf$, and $C$ is a generic, universal constant we will not care to change. From Gronwall's lemma, we can have
\begin{equation}\label{dercero}
|\by(t)-\bz(t)|^2\le C(E(\by)+E(\bz))e^{CM^2T}
\end{equation}
for all $t\in[0, T]$. This means
\begin{gather}
\|\by-\bz\|_{L^\infty(0, T; \R^N)}\le e^{CM^2T/2}\sqrt{C(E(\by)+E(\bz))},\nonumber\\
\|\by-\bz\|_{L^2(0, T; \R^N)}^2\le CTe^{CM^2T}(E(\by)+E(\bz)).\label{cota}
\end{gather}
But once we can rely on this information,  the above decomposition allows us to write in a similar manner
$$
\int_0^t |\by'(s)-\bz'(s)|^2\,ds\le C(E(\by)+E(\bz))+CM^2\int_0^t|\by(s)-\bz(s)|^2\,ds
$$
and
$$
\|\by'-\bz'\|^2_{L^2(0, T; \R^N)}\le C(E(\by)+E(\bz))+CM^2\|\by-\bz\|^2_{L^2(0, T; \R^N)},
$$
and thus, taking into account \eqref{cota}, 
$$
\|\by'-\bz'\|^2_{L^2(0, T; \R^N)}\le (C+C^2M^2Te^{CM^2T})(E(\by)+E(\bz)).
$$
Our estimate \eqref{enhcoer} is then a consequence that the norm in $\bbH$ can be taken to be the $L^2$-norm of the derivative. 
\end{proof}

The second basic ingredient is \eqref{comporinff}. We will be using Proposition \ref{principalbis}. We assume further that the mapping $\bbf$ is smooth with a derivative uniformly bounded to guarantee the uniform Lipschitz condition. 
For the operator $\T$, we will put $\bZ=\T\bz$ for $\bz\in\bbH$, and linearize \eqref{ode} at the path $\bx_0+\bz(t)$ to write
\begin{gather}
\bZ'(t)=\bbf(\bx_0+\bz(t))+\nabla\bbf(\bx_0+\bz(t))\bZ(t)-\bz'(t)\hbox{ in }[0, T],\label{linealizacion}\\ 
\bZ(0)=\bcero.\nonumber
\end{gather}
This is a linear, differential, non-constant coefficient system for $\bZ$ with coefficients depending on $\bz$. 
Under smoothness assumptions, which we take for granted, such operator $\T$ is locally Lipschitz because the image $\bZ=\T\bz$ is defined through a linear initial-value, Cauchy problem with coefficients depending continuously on $\bz$. 

The important property to be checked, concerning $\T$, is \eqref{propoper}. It is elementary to see, under smoothness assumptions which, as indicated, we have taken for granted, that
$$
\langle E'(\bz), \bZ\rangle=
\int_0^T(\bz'(s)-\bbf(\bx_0+\bz(s)))(\bZ'(s)-\nabla\bbf(\bx_0+\bz(s))\bZ(s))\,ds
$$
Hence, for $\bZ=\T\bz$ coming from \eqref{linealizacion}, we immediately deduce that
$$
\langle E'(\bz), \bZ\rangle=-2E(\bz).
$$
We are, then, entitled to apply Proposition \ref{principalbis} to conclude that functional $E$ in 
\eqref{erroredo} is an error functional after Definition \ref{errorgeneral}, and we are entitled to utilize Proposition \ref{basicas} to conclude the following.
\begin{theorem}
If the mapping $\bbf(\by):\R^N\to\R^N$ is $\C^1$- with a globally bounded gradient, then, for arbitrary $\bx_0\in\R^N$ and $T>0$, problem \eqref{ode} admits a unique $\C^1$-solution 
$$
\overline\bx(t):[0, T)\to\R^N,
$$ 
and there is a positive constant $C$ such that
$$
\|\bx-\overline\bx\|_\bbH^2\le C\int_0^T|\bx'(s)-\bbf(\bx_0+\bx(s))|^2\,ds
$$
for every $\bx\in\bbH$.
\end{theorem}
There is no difficulty in showing a local version of this result by using the same ideas. 

% ================================================
\section{Linear hyperbolic example}
% ================================================
Since most likely readers will not be used to think about hyperbolic problems in these terms, we will treat the most transparent example of a linear, hyperbolic problem from this perspective, and later apply the method to a non-linear wave equation.

 We seek a (weak) solution $u(t, \bx)$ of some sort of the problem
\begin{gather}
u_{tt}(t, \bx)-\Delta u(t, \bx)-u(t, \bx)=f(t, \bx)\hbox{ in }\R^N_+,\label{ondas}\\ u(0, \bx)=0, u_t(0, \bx)=0\hbox{ on }t=0,\nonumber
\end{gather}
for $f\in L^2(\R^N_+)$. Here $\R^N_+$ is the upper half hyperspace $[0, +\infty)\times\R^N$. We look for 
$$
u(t, \bx)\in H^1_0(\R^N_+)
$$ 
(jointly in time and space) such that
\begin{gather}
\int_{\R^N_+}[u_t(t, \bx)w_t(t, \bx)-\nabla u(t, \bx)\cdot\nabla w(t, \bx)\label{formadeb}\\
+(f(t, \bx)+u(t, \bx))w(t, \bx)]\,d\bx\,dt=0\nonumber
\end{gather}
for every test function
$$
w(t, \bx)\in H^1(\R^N_+).
$$ 
Note how the arbitrary values of the test function $w$ for $t=0$ imposes the vanishing initial velocity $u_t(0, \bx)=0$. 

To setup a suitable error functional 
$$
E(u):  H^1_0(\R^N_+)\to\R^+
$$
for every $u(t, \bx)$, 
and not just for the solution we seek, we utilize a natural least-square concept as indicated in the Introduction. Define an appropriate defect or residual function 
$$
U(t, \bx)\in H^1(\R^N_+),
$$ 
for each such $u\in H^1_0(\R^N_+)$, as the unique variational solution of
\begin{equation}\label{debilondas}
\int_{\R^N_+}[(u_t+U_t)w_t-(\nabla u-\nabla U)\cdot\nabla w+(f+u+U) w]\,d\bx\,dt=0
\end{equation}
valid for every $w\in H^1(\R^N_+)$. This function $U$ is indeed the unique minimizer over $H^1(\R^N_+)$ of the strictly convex, quadratic functional
$$
I(w)=\int_{\R^N_+}\left(\frac12[(w_t+u_t)^2+|\nabla w-\nabla u|^2+(u+w)^2]+fw\right)\,d\bx\,dt
$$
for each fixed $u$. 
The size of $U$ is regarded as a measure of the departure of $u$ from being a solution of our problem
\begin{gather}
E:H^1_0(\R^N_+)\to\R^+,\nonumber\\ 
E(u)=\int_{\R^N_+}\frac12[U^2_t(t, \bx)+|\nabla U(t, \bx)|^2+U^2(t, \bx)]\,d\bx\,dt.\label{errorhyp}
\end{gather}
We can also put, in a short form, 
\begin{equation}\label{sencillo}
E(u)=\frac12\|U\|_{H^1(\R^N_+)}^2;
\end{equation}
or even
$$
E(u)=\frac12\|u_{tt}-\Delta u-u-f\|^2_{H^{-1}(\R^N_+)},
$$
though we will stick to \eqref{sencillo} to better manipulate $E$. 

We would like to apply Proposition \ref{basicas} in this situation, and hence, we set to ourselves the task of checking the two main assumptions in Definition \ref{errorgeneral}.
Our functional $E$ is definitely smooth and non-negative to begin with. 

It is not surprising that in order to work with the wave equation the following two linear operators 
\begin{gather}
\bbS:H^1(\R^N_+)\mapsto H^1(\R^N_+),\quad \bbS w(t, \bx)=w(t, -\bx),\nonumber\\
\cS: H^1(\R^N_+)\mapsto H^1(\R^N_+)^*,\nonumber\\
\cS u(t, \bx)=(u(t, -\bx), u_t(t, -\bx), \nabla u(t, -\bx)),\nonumber
\end{gather}
will play a role. $H^1(\R^N_+)^*$ is here the dual space of $H^1(\R^N_+)$, not to be mistaken with $H^{-1}(\R^N_+)$. 
Put $\bbH=\cS(H^1_0(\R^N_+))$. The following fact is elementary. Check for instance \cite{brezis}. 

\begin{lemma}\label{hiperbolico}
\begin{enumerate}
\item The map $\bbS$ is an isometry.
\item $\bbH$ is a closed subspace of $H^1(\R^N_+)^*$, and 
$$
\cS:H^1_0(\R^N_+)\to\bbH
$$ 
is a bijective, continuous mapping. In fact, we clearly have
\begin{equation}\label{coer}
\|u\|_{H^1_0(\R^N_+)}\le \|\cS u\|_{H^1(\R^N_+)^*}.
\end{equation}
\end{enumerate}
\end{lemma}

We can now proceed to prove inequality \eqref{enhcoer} in this new context.
\begin{proposition}\label{enhcoerhyp}
There is a constant $K>0$ such that
$$
\|u-v\|_{H^1_0(\R^N_+)}^2\le K(E(u)+E(v)),
$$
for every pair $u, v\in H^1_0(\R^N_+)$.
\end{proposition}
\begin{proof}
Let $U, V\in H^1(\R^N)$ be the respective residual functions associated with $u$ and $v$. Because we are in a linear situation, if we replace
$$
u-v\mapsto u,\quad U-V\mapsto U,
$$
we would have
\begin{equation}\label{debilcero}
\int_{\R^N_+}[(u_t+U_t)w_t-(\nabla u-\nabla U)\cdot\nabla w+(u+U) w]\,d\bx\,dt=0,
\end{equation}
for every $w\in H^1(\R^N_+)$. 
If we use $\bbS w$ in \eqref{debilcero} instead of $w$, we immediately find
\begin{gather}
\int_{\R^N_+}[(u_t+U_t)w_t(t, -\bx)+(\nabla u-\nabla U)\cdot\nabla w(t, -\bx)\nonumber\\
+(u+U)w(t, -\bx)]\,d\bx\,dt=0.\nonumber
\end{gather}
The terms involving $U$ can be written in compact form as
$$
\langle U, \bbS w\rangle_{\langle H^1(\R^N_+), H^1(\R^N_+)\rangle}
$$
while a natural change of variables in the terms involving $u$ leads to writing these in the form
$$
\langle w, \cS u\rangle_{\langle H^1(\R^N_+), H^1(\R^N_+)^*\rangle}.
$$
Hence, for every $w\in H^1(\R^N_+)$, we find
$$
\langle U, \bbS w\rangle_{\langle H^1(\R^N_+), H^1(\R^N_+)\rangle}+
\langle w, \cS u\rangle_{\langle H^1(\R^N_+), H^1(\R^N_+)^*\rangle}=0.
$$
Bearing in mind this identity, we have, through the  Lemma \ref{hiperbolico}, 
\begin{align}
\|u\|_{H^1_0(\R^N_+)}\le&\|\cS u\|_{H^1(\R^N_+)^*}\nonumber\\
=&\sup_{\|w\|_{H^1(\R^N_+)}\le 1}\langle w, \cS u\rangle_{\langle H^1(\R^N_+), H^1(\R^N_+)^*\rangle}\nonumber\\
\le&\|U\|_{H^1(\R^N_+)}\,\|w\|_{H^1(\R^N_+)}
\nonumber\\
\le&\|U\|_{H^1(\R^N_+)}
.\nonumber
\end{align}
If we go back to 
$$
u\mapsto u-v,\quad U\mapsto U-V,
$$
we are led to
\begin{align}
\|u-v\|^2_{H^1_0(\R^N_+)}\le &\|U-V\|^2_{H^1(\R^N_+)}\nonumber\\
\le &C\left(\|U\|^2_{H^1(\R^N_+)}+\|V\|^2_{H^1(\R^N_+)}\right)\nonumber\\
\le &C\left(E(u)+E(v)\right),\nonumber
\end{align}
for some constant $C>0$. 
\end{proof}
The second main ingredient to apply Proposition \ref{basicas} is to show that 
$E$ defined in \eqref{errorhyp} complies with \eqref{comporinff} too. To this end, we need to compute the derivative $E'(u)$, and so 
we perform a perturbation
$$
u+\epsilon v\mapsto U+\epsilon V,
$$
in \eqref{debilondas} to write
\begin{gather}
\int_{\R^N_+}[(u_t+\epsilon v_t+U_t+\epsilon V_t)w_t-(\nabla u+\epsilon\nabla v-\nabla U-\epsilon\nabla V)\cdot\nabla w\nonumber\\
+(f+u+\epsilon v+U+\epsilon V) w]\,d\bx\,dt=0.\nonumber
\end{gather}
The term to order 1 in $\epsilon$ should vanish
\begin{equation}\label{firstorder}
\int_{\R^N_+}[(v_t+V_t)w_t-(\nabla v-\nabla V)\cdot\nabla w+(v+V) w]\,d\bx\,dt=0
\end{equation}
for every $w\in H^1(\R^N_+)$. On the other hand, by differentiating
$$
E(u+\epsilon v)=\int_{\R^N_+}\frac12((U+\epsilon V)^2_t+|\nabla U+\epsilon\nabla V|^2+(U+\epsilon V)^2)\,d\bx\,dt,
$$
with respect to $\epsilon$, and setting $\epsilon=0$, we arrive at
$$
\langle E'(u), v\rangle=\int_{\R^N_+}(U_tV_t+\nabla U\cdot\nabla V+UV)\,d\bx\,dt.
$$
By taking $w=U$ in \eqref{firstorder}, we can also write
\begin{align}
\langle E'(u), v\rangle=&\int_{\R^N_+}(\nabla v\cdot\nabla U-v_tU_t-vU)\,d\bx\,dt\nonumber\\
=&-\langle \bbS v, \cS U\rangle_{\langle H^1(\R^N_+), H^1(\R^N_+)^*\rangle}.\nonumber
\end{align}
From this identity, which ought to be valid for every $v\in H^1_0(\R^N_+)$, we clearly conclude that if 
$E'(u)\to\bcero$ then $\cS U\to\bcero$ as well, because $\bbS$ preserves the norm. Realizing that 
$$
E(u)=\frac12\|U\|_{H^1(\R^N_+)}^2\le \frac12\|\cS U\|_{H^1(\R^N_+)^*}^2,
$$
by estimate \eqref{coer}, we conclude the following.
\begin{proposition}
The functional $E$ in \eqref{errorhyp} is an error functional in the sense of Definition \ref{errorgeneral}.
\end{proposition}

Our main abstract result Proposition \ref{basicas} applies in this situation too, and we can conclude
\begin{theorem}
Problem \eqref{ondas} admits a unique weak solution $u\in H^1_0(\R^N_+)$ in the sense \eqref{formadeb},
and for every other $v\in H^1_0(\R^N_+)$, we have
$$
\|u-v\|_{H^1_0(\R^N_+)}^2\le KE(v),
$$
for some positive constant $K$. 
\end{theorem}

% ================================================
\section{Non-linear monotone problems}\label{seven}
Suppose we would like to solve, or approximate the solution of, a certain non-linear elliptic system of PDEs of the form
$$
\dv[\Phi(\nabla u)]=0\hbox{ in }\Omega,\quad u=u_0\hbox{ on }\partial\Omega,
$$
for a non-linear, smooth map
$$
\Phi(\ba):\R^N\to\R^N.
$$
$\Omega\subset\R^N$ is assumed to be a regular, bounded domain. 
One can set up a natural, suitable, non-negative, smooth functional based on the least-squares idea, as already introduced, 
\begin{equation}\label{funcellnl}
E(v):H^1_0(\Omega)\to\R
\end{equation}
by putting
$$
E(v)=\frac12\int_\Omega|\nabla U(\bx)|^2\,d\bx
$$
where
$$
\dv[\Phi(\nabla v+\nabla u_0)+\nabla U]=0\hbox{ in }\Omega
$$
and $U\in H^1_0(\Omega)$. We can also put
$$
E(v)=\frac12\|\dv[\Phi(\nabla v+\nabla u_0)]\|^2_{H^{-1}(\Omega)}.
$$
Our goal is to apply again Proposition \ref{basicas}, or, since we are now in a non-linear situation, Proposition \ref{principalbis}. Anyhow, \eqref{enhcoer} is necessary. 

\begin{lemma}
Let $\Phi(\ba):\R^N\to\R^N$ be a smooth-map with linear growth at infinity, i.e.
\begin{equation}\label{lineargr}
|\Phi(\ba)|\le C_1|\ba|+C_0,
\end{equation}
with $C_1>0$, and strictly monotone in the sense
\begin{equation}\label{monotonia}
(\Phi(\ba_1)-\Phi(\ba_0))\cdot(\ba_1-\ba_0)\ge c|\ba_1-\ba_0|^2,\quad c>0,
\end{equation}
for every pair of vectors $\ba_i$, $i=0, 1$. Then there is a positive constant $C$ such that
$$
\|u-v\|_{H^1_0(\Omega)}^2\le C(E(u)+E(v)),
$$
for every pair $u, v\in H^1_0(\Omega)$. 
\end{lemma}
\begin{proof}
Let $u, v\in H^1_0(\Omega)$, and let $U, V\in H^1_0(\Omega)$ be their respective residuals in the sense 
\begin{equation}\label{sistemas}
\dv[\Phi(\nabla u+\nabla u_0)+\nabla U]=0,\quad 
\dv[\Phi(\nabla v+\nabla u_0)+\nabla V]=0
\end{equation}
in $\Omega$, and
$$
E(u)=\frac12\|\nabla U\|^2_{L^2(\Omega; \R^N)},\quad E(v)=\frac12\|\nabla V\|^2_{L^2(\Omega; \R^N)}.
$$
If we use $u-v$ as test field in \eqref{sistemas}, we find
\begin{gather}
\int_\Omega\Phi(\nabla u+\nabla u_0)\cdot(\nabla u-\nabla v)\,d\bx=-\int_\Omega\nabla U\cdot(\nabla u-\nabla v)\,d\bx,\nonumber\\
 \int_\Omega\Phi(\nabla v+\nabla u_0)\cdot(\nabla u-\nabla v)\,d\bx=-\int_\Omega\nabla V\cdot(\nabla u-\nabla v)\,d\bx.\nonumber
\end{gather}
The monotonicity condition, together with these identities, takes us, by subtracting one from the other, to
$$
c\int_\Omega|\nabla u-\nabla v|^2\,d\bx\le\int_\Omega(\nabla V-\nabla U)\cdot(\nabla u-\nabla v)\,d\bx.
$$
The standard Cauchy-Schwarz inequality implies that
$$
c\|\nabla u-\nabla v\|_{L^2(\Omega; \R^N)}\le \|\nabla U-\nabla V\|_{L^2(\Omega; \R^N)},
$$
and thus, taking into account the triangular inequality, we have
$$
c^2\|\nabla u-\nabla v\|_{L^2(\Omega; \R^N)}^2\le 4E(u)+4E(v).
$$
The use of Poincaré's inequality yields our statement. 
\end{proof}
The second ingredient, to apply Theorem \ref{principalbis}, is the operator $\T$ which comes directly from linearization or from Newton's method. Given an approximation of the solution $v+u_0$, $v\in H^1_0(\Omega)$, we seek a better approximation $V\in H^1_0(\Omega)$ in the form
\begin{equation}\label{sistemalineal}
\dv[\Phi(\nabla v+\nabla u_0)+\nabla\Phi(\nabla v+\nabla u_0)\nabla V]=0\hbox{ in }\Omega,
\end{equation}
as a linear approximation of 
$$
\dv[\Phi(\nabla v+\nabla u_0+\nabla V)]=0\hbox{ in }\Omega.
$$
We therefore define
\begin{equation}\label{newoper}
\T:H^1_0(\Omega)\to H^1_0(\Omega),\quad \T v=V,
\end{equation}
where $V$ is the solution of \eqref{sistemalineal}. The fact that $\T$ is well-defined is a direct consequence of the standard Lax-Milgram lemma and the identification
$$
\bA=\nabla\Phi(\nabla v+\nabla u_0),\quad \ba=\Phi(\nabla v+\nabla u_0),
$$
provided
$$
|\nabla\Phi(\bv)|\le M,\quad 
 \bu\cdot\nabla\Phi(\bv)\bu\ge c|\bu|^2,\quad M, c>0.
 $$
The first bound is compatible with linear growth at infinity, \eqref{lineargr}, while the second one is a consequence of monotonicity \eqref{monotonia}. On the other hand, the smoothness of $\T$ depends directly on the smoothness of $\Phi$, specifically, we assume $\Phi$ to be $\C^2$. Since $\T$ comes from Newton's method, condition \eqref{propoper} is guaranteed. We are hence entitled to apply Proposition \ref{principalbis} and conclude that

\begin{theorem}
Let $\Phi(\ba):\R^N\to\R^N$ be a $\C^2$-mapping such that
\begin{gather}
|\nabla\Phi(\ba)|\le M,\nonumber\\
 (\Phi(\ba_1)-\Phi(\ba_0))\cdot(\ba_1-\ba_0)\ge c|\ba_1-\ba_0|^2,\nonumber
 \end{gather}
 for constants $M, c>0$, and every $\ba$, $\ba_1$, $\ba_0$ in $\R^N$. 
There is a unique weak solution $u\in u_0+H^1_0(\Omega)$, for arbitrary $u_0\in H^1(\Omega)$, of the equation
$$
\dv[\Phi(\nabla u)]=0\hbox{ in }\Omega,\quad u-u_0\in H^1_0(\Omega).
$$
Moreover
$$
\|u-v\|^2_{H^1_0(\Omega)}\le CE(v)
$$
for every other $v\in u_0+H^1_0(\Omega)$.
\end{theorem}
It is not hard to design appropriate sets of assumptions to deal with more general equations of the form
$$
\dv[\Phi(\nabla v(\bx), v(\bx), \bx)]=0. 
$$

%====================================================
\section{Non-linear waves}
%====================================================
We would like to explore non-linear equations of the form
$$
u_{tt}(t, \bx)-\Delta u(t, \bx)-f(\nabla u(t, \bx), u_t(t, \bx), u(t, \bx))=0\hbox{ in }(t, \bx)\in\R^N_+,
$$
subjected to initial conditions 
$$
u(0, \bx)=u_0(\bx),\quad u_t(0, \bx)=u_1(\bx)
$$
for appropriate data $u_0$ and $u_1$ belonging to suitable spaces to be determined. 
Dimension $N$ is taken to be at least two. 
Though more complicated situations could be considered allowing for a monotone main part in the equation, as in the previous section, to better understand the effect of the term incorporating lower-order terms, we will restrict ourselves to the equation above.  Conditions on the non-linear term 
$$
f(\bz, z, u):\R^N\times\R\times\R\to\R
$$
will be specified along the way as needed.

Our ambient space will be $H^1(\R^N_+)$ so that weak solutions $u$ are sought in $H^1(\R^N_+)$. If we assume  
$$
u_0\in H^{1/2}(\R^N),\quad u_1\in L^2(\R^N),
$$ 
we can take for granted, without loss of generality, that both $u_0$ and $u_1$ identically vanish and $u\in H^1_0(\R^N_+)$, at the expense of permitting 
$$
f(\bz, z, u, t, \bx):\R^N\times\R\times\R\times\R\times\R^N\to\R.
$$
We will therefore stick to the problem
\begin{equation}\label{nonlinw}
u_{tt}-\Delta u-f(\nabla u, u_t, u, t, \bx)=0\hbox{ in }(t, \bx)\in\R^N_+,
\end{equation}
subjected to initial conditions 
\begin{equation}\label{iniccond}
u(0, \bx)=0,\quad u_t(0, \bx)=0.
\end{equation}

A weak solution $u\in H^1_0(\R^N_+)$ of \eqref{nonlinw} is such that
\begin{equation}\label{formadebilnlw}
\int_{\R^N_+}[-u_tw_t+\nabla u\cdot\nabla w-f(\nabla u, u_t, u, t, \bx)w]\,d\bx\,dt=0
\end{equation}
for every $w\in H^1(\R^N_+)$. This weak formulation asks for the non-linear term
recorded in the function $f$ to comply with
\begin{equation}\label{nonlinterm}
|f(\bz, z, u, t, \bx)|\le C(|\bz|+|z|+|u|^{(N+1)/(N-1)})+f_0(t, \bx)
\end{equation}
for a function $f_0\in L^2(\R^N_+)$, in such a way that the composition
$$
f(\nabla u, u_t, u, t, \bx)\in L^2(\R^N_+)
$$
for every $u\in H^1(\R^N_+)$. As expected, for every $u\in H^1_0(\R^N_+)$ we define its residual $U\in H^1(\R^N_+)$ through
\begin{equation}\label{resnlw}
\int_{\R^N_+}[(U_t+u_t)w_t-(\nabla u-\nabla U)\cdot\nabla w+(U+f(\nabla u, u_t, u, t, \bx))w]\,d\bx\,dt=0
\end{equation}
which ought to be correct for every test $w\in H^1(\R^N_+)$; and the functional
\begin{equation}\label{funcerrnlw}
E(u):H^1_0(\R^N_+)\to\R^+,\quad E(u)=\frac12\|U\|^2_{H^1(\R^N_+)},
\end{equation}
as a measure of departure of $u$ from being a weak solution of \eqref{nonlinw}. 
Note how \eqref{resnlw} determines $U$ in a unique way. Indeed, such $U$ is the unique minimizer of the strictly convex, quadratic functional
$$
\frac12\int_{\R^N_+}\left[(U_t+u_t)^2+|\nabla U-\nabla u|^2+(U+f(\nabla u, u_t, u, t, \bx))^2\right]\,d\bx\,dt
$$
define for $U\in H^1(\R^N_+)$. 
We claim that under appropriate additional hypotheses, we can apply Proposition \ref{principalbis} to this situation. To explain things in the most affordable way, however, we will show that Proposition \ref{basicas} can also be applied directly. This requires to check that $E$ in \eqref{funcerrnlw} is indeed an error functional in the sense of Definition \ref{errorgeneral}.

We will be using the operators and the formalism right before Lemma \ref{hiperbolico}, as well as bound \eqref{coer} in this lemma. 

\begin{lemma}
Suppose the function $f(\bz, z, u, t, \bx)$ is such that
\begin{enumerate}
\item $f(\bcero, 0, 0, t, \bx)\in L^2(\R^N_+)$;
\item the difference $f(\bz, z, u, t, \bx)-u$ is globally Lipschitz with respect to triplets $(\bz, z, u)$ in the sense
\begin{gather}
|f(\bz, z, u, t, \bx)-u-f(\by, y, v, t, \bx)+v|\le \nonumber\\
M\left(|\bz-\by|+|z-y|+\frac1D|u-v|^{(N+1)/(N-1)}\right),\nonumber
\end{gather}
where $D$ is the constant of the corresponding embedding 
$$
H^1(\R^N_+)\subset L^{2(N+1)/(N-1)}(\R^N_+),
$$ 
and $M<1$.
\end{enumerate}
Then there is a positive constant $K$ with
$$
\|u-v\|_{H^1_0(\R^N_+)}^2\le K(E(u)+E(v)),
$$
for every pair $u, v\in H^1_0(\R^N_+)$, where $E$ is given in \eqref{funcerrnlw}. 
\end{lemma}
\begin{proof}
Note how our hypotheses on the nonlinearity $f$ imply the bound \eqref{nonlinterm} by taking 
$$
\by=\bcero,\quad y=v=0,\quad f_0(t, \bx)=f(\bcero, 0, 0, t, \bx).
$$ 

If $u$, $v$ belong to $H^1_0(\R^N_+)$, and $U$, $V$ in $H^1(\R^N_+)$ are their respective residuals, then
\begin{gather}
\int_{\R^N_+}[(U_t+u_t)w_t-(\nabla u-\nabla U)\cdot\nabla w+(U+f(\nabla u, u_t, u, t, \bx))w]\,d\bx\,dt=0,\nonumber\\
\int_{\R^N_+}[(V_t+v_t)w_t-(\nabla v-\nabla V)\cdot\nabla w+(V+f(\nabla v, v_t, v, t, \bx))w]\,d\bx\,dt=0,\nonumber
\end{gather}
for every $w\in H^1(\R^N_+)$. By subtracting one from the other, and letting 
$$
s=u-v,\quad S=U-V,
$$
we find 
\begin{gather}
\int_{\R^N_+}[(S_t+s_t)w_t-(\nabla s-\nabla S)\cdot\nabla w+\nonumber\\
(S+f(\nabla u, u_t, u, t, \bx)-f(\nabla v, v_t, v, t, \bx))w]\,d\bx\,dt=0,\nonumber
\end{gather}
for every $w\in H^1(\R^N_+)$. We can recast this identity, by using the formalism in the corresponding linear situation around Lemma \ref{hiperbolico}, as
\begin{gather}
\langle S, \bbS w\rangle_{\langle H^1(\R^N_+), H^1(\R^N_+)\rangle}+
\langle w, \cS s\rangle_{\langle H^1(\R^N_+), H^1(\R^N_+)^*\rangle}\nonumber\\
=-\int_{\R^N_+}[f(\nabla u, u_t, u, t, \bx)-f(\nabla v, v_t, v, t, \bx)-s)w]\,d\bx\,dt.
\nonumber
\end{gather}
The same manipulations as in the proof of Proposition  \ref{enhcoerhyp}, together with the assumed Lipschitz property on $f$,  lead immediately to
\begin{align}
\|s\|_{H^1_0(\R^N_+)}\le &\|\cS s\|_{H^1(\R^N_+)^*}\nonumber\\
=&\sup_{\|w\|_{H^1(\R^N_+)}\le 1}|\langle w, \cS s\rangle_{\langle H^1(\R^N_+), H^1(\R^N_+)^*\rangle}|\nonumber\\
=&\sup_{\|w\|_{H^1(\R^N_+)}\le 1}\left|\langle S, \bbS w\rangle+\int_{\R^N_+}(f_u-f_v-s)w\,d\bx\,dt\right|\nonumber\\
\le&\|S\|_{H^1(\R^N_+)}\,\|w\|_{H^1(\R^N_+)}+M\|s\|_{H^1_0(\R^N_0)}\|w\|_{H^1(\R^N_+)}\nonumber\\
\le&\|S\|_{H^1(\R^N_+)}+M\|s\|_{H^1_0(\R^N_+)}.\nonumber
\end{align}
We are putting
$$
f_u=f(\nabla u, u_t, u, t, \bx),\quad f_v=f(\nabla v, v_t, v, t, \bx),
$$
for the sake of notation. Note also the use of the embedding constant. 
The resulting final inequality, and the relative sizes of these constants, show our claim. 
\end{proof}

We turn to the second important property for $E$ to become an error functional, namely,
$$
\lim_{E'(\bu)\to\bcero}E(\bu)=0
$$
over bounded subsets of $H^1_0(\R^N_+)$. We assume that the non-linearity $f$ is $C^1$- with respect to $(\bz, z, u)$, and its partial derivatives are uniformly bounded. 
To compute the derivative $E'(u)$ at an arbitrary $u\in H^1_0(\R^N_+)$, we perform, as usual, the perturbation to first-order
$$
u\mapsto u+\epsilon v,\quad U\mapsto U+\epsilon V,
$$
 and introduce them in \eqref{resnlw}. After differentiation with respect to $\epsilon$, and setting $\epsilon=0$, we find
 \begin{equation}\label{perturbacion}
 \int_{\R^N_+}[(V_t+v_t)w_t-(\nabla v-\nabla V)\cdot\nabla w+(V+\overline f_\bz\cdot\nabla v+\overline f_z v_t+\overline f_u v)w]\,d\bx\,dt=0,
\end{equation}
for all $w\in H^1(\R^N_+)$, where
$$
\overline f_\bz(t, \bx)=f_\bz(\nabla u(t, \bx), u_t(t, \bx), u(t, \bx), t, \bx),
$$
and the same for $\overline f_z(t, \bx)$ and $\overline f_u(t, \bx)$. On the other hand, 
$$
\langle E'(u), v\rangle=\lim_{\epsilon\to0}\frac1\epsilon(E(u+\epsilon v)-E(u))
$$
is clearly given by
$$
\langle E'(u), v\rangle=\int_{\R^N_+}\left(\nabla U\cdot\nabla V+U_t\,V_t+U\,V\right)\,d\bx\,dt.
$$
If we use $w=U$ in \eqref{perturbacion}, we can write
$$
\langle E'(u), v\rangle=\int_{\R^N_+}\left[-v_t\,U_t+\nabla v\cdot\nabla U-(\overline f_\bz\cdot\nabla v+\overline f_z v_t+\overline f_u v)U\right]\,d\bx\,dt.
$$
The validity of this representation for every $v\in H^1_0(\R^N_+)$ enables us to identify $E'(u)$ with the triplet 
$$
(-U_t-\overline f_zU, \nabla U-U\overline f_\bz, -\overline f_uU),
$$
in the sense $E'(u)=\cS U$ where the linear operator
$$
\cS:H^1(\R^N_+)\mapsto H^{-1}(\R^N_+)
$$
is precisely determined by
$$
\langle\cS U, v\rangle=\int_{\R^N_+}\left[-v_t\,U_t+\nabla v\cdot\nabla U-(\overline f_\bz\cdot\nabla v+\overline f_z v_t+\overline f_u v)U\right]\,d\bx\,dt
$$
for every $v\in H^1_0(\R^N_+)$. Notice how this operator $\cS$ is well-defined because the non-linearity $f$ has been assumed to be globally Lipschitz with partial derivatives uniformly bounded. To conclude that $E'(u)\to\bcero$ implies $U\to0$ and, hence $E(u)=0$, we need to ensure that this operator $\cS$ is injective. We conjecture that this is so, without further requirements; but to simplify the argument here, we add the assumption that 
$$
|f_u(\bz, z, u, t, z)|\ge \epsilon>0
$$
for every $(\bz, z, u, t, z)$. Under this additional hypothesis, the condition
$$
(-U_t-\overline f_zU, \nabla U-U\overline f_\bz, -\overline f_uU)=\bcero
$$
automatically implies 
$$
U=\nabla U=U_t=0
$$
and hence $E(u)=0$. 

\begin{theorem}
Suppose the non-linearity $f(\bz, z, u, t, \bx)$ is $\C^1$- with respect to variables $(\bz, z, u)$, and:
\begin{enumerate}
\item $f(\bcero, 0, 0, t, \bx)\in L^2(\R^N_+)$;
\item the difference $f(\bz, z, u, t, \bx)-u$ is Lipschitz with respect to triplets $(\bz, z, u)$ in the sense
\begin{gather}
|f(\bz, z, u, t, \bx)-u-f(\by, y, v, t, \bx)+v|\le \nonumber\\
M\left(|\bz-\by|+|z-y|+\frac1D|u-v|^{(N+1)/(N-1)}\right),\nonumber
\end{gather}
where $D$ is the constant of the corresponding embedding 
$$
H^1(\R^N_+)\subset L^{2(N+1)/(N-1)}(\R^N_+),
$$ 
and $M<1$;
\item non-vanishing of $f_u$: there is some $\epsilon>0$ with 
$$
|f_u(\bz, z, u, t, z)|\ge \epsilon>0.
$$
\end{enumerate}
Then the problem
$$
u_{tt}-\Delta u+f(\nabla u, u_t, u, t, \bx)=0\hbox{ in }(t, \bx)\in\R^N_+
$$ 
under vanishing initial conditions 
$$
u(0, \bx)=u_t(0, \bx)=0,\quad \bx\in\R^N,
$$
admits a unique weak solution $u\in H^1_0(\R^N_+)$ in the sense \eqref{formadebilnlw}, and
$$
\|v-u\|^2_{H^1_0(\R^N_+)}\le K E(v),
$$
for any other $v\in H^1_0(\R^N_+)$.
\end{theorem}

Without the global Lipschitzianity condition on $f$ in the previous statement, but only the smoothness with respect to triplets $(\bz, z, u)$, 
only a local existence result is possible. This is standard. 

\section{The steady Navier-Stokes system}
For a bounded, Lipschitz, connected domain $\Omega\subset\R^N$, $N=2, 3$, we are concerned with the steady Navier-Stokes system
\begin{equation}\label{navsto}
-\nu\Delta\bu+\nabla\bu\,\bu+\nabla u=\bbf,\quad \dv\bu=0\hbox{ in }\Omega,
\end{equation}
for a vector field $\bu\in H^1_0(\Omega; \R^N)$, and a scalar, pressure field $u\in L^2(\Omega)$. The external force field $\bbf$ is assumed to belong to the dual space $H^{-1}(\Omega; \R^N)$. The parameter $\nu>0$ is viscosity. 
Because of the incompressibility condition, the system can also be written in the form
$$
-\nu\Delta\bu+\dv(\bu\otimes\bu)+\nabla u=\bbf,\quad \dv\bu=0\hbox{ in }\Omega.
$$
A weak solution is a divergence-free vector field $\bu\in H^1_0(\Omega; \R^N)$, and a scalar field $u\in L^2(\Omega)$, normalized by demanding vanishing average in $\Omega$,  such that
$$
\int_\Omega [\nu\nabla\bu(\bx):\nabla\bv(\bx)-\bu(\bx)\nabla\bv(\bx)\bu(\bx)-u(\bx)\dv\bv(\bx)]\,d\bx=\langle\bbf, \bv\rangle
$$
where the right-hand side stands for the duality pairing 
$$
H^{-1}(\Omega; \R^N)-H^1_0(\Omega; \R^N).
$$ 

We propose to look at this problem incorporating the incompressibility constraint into the space as part of feasibility as is usually done; there is also the alternative to treat the same situation incorporating a penalization on the divergence into the functional, instead of including it into the class of admissible fields (see \cite{lemmunped}). The pressure field rises as the multiplier corresponding to the divergence-free constraint. 

Let 
$$
\bbD\equiv H^1_{0, div}(\Omega; \R^N)=\{\bu\in H^1_0(\Omega; \R^N): \dv\bu=0\hbox{ in }\Omega\}.
$$
For every such $\bu$, we determine its residual $\bU$, in a unique way, as the solution of the restricted variational problem
$$
\hbox{Minimize in }\bV\in \bbD:\quad \int_\Omega\left[\frac12|\nabla\bV|^2+(\nu\nabla\bu-\bu\otimes\bu):\nabla\bV\right]\,d\bx-\langle\bbf, \bV\rangle.
$$
The pressure $v$ comes as the corresponding multiplier for the divergence-free constraint, in such a way that the unique minimizer $\bU$ is determined through the variational equality
$$
\int_\Omega (\nabla\bU:\nabla\bV+(\nu\nabla\bu-\bu\otimes\bu):\nabla\bV+v\dv\bV)\,d\bx-\langle\bbf, \bV\rangle
$$
valid for every test field $\bV\in H^1_0(\Omega; \R^N)$. This is the weak form of the optimality condition associated with the previous variational problem
\begin{equation}\label{residuo}
-\Delta\bU-\nu\Delta\bu+\dv(\bu\otimes\bu)-\bbf+\nabla v=\bcero\hbox{ in }\Omega,
\end{equation}
for $\bU\in\bbD$. The multiplier $v\in L^2_0(\Omega)$ (square-integrable fields with a vanishing average) is the pressure. We define
\begin{equation}\label{funcionalns}
E(\bu):\bbD\to\R^+,\quad E(\bu)=\frac12\int_\Omega|\nabla\bU(\bx)|^2\,d\bx.
\end{equation}

For $\bu, \bv\in\bbD$, let $\bU, \bV\in\bbD$ be their respective residuals, and $u, v$ their respective pressure fields. Put 
$$
\bw=\bu-\bv\in\bbD,\quad \bW=\bU-\bV\in\bbD, \quad w=u-v\in L^2(\Omega).
$$
It is elementary to find, by subtraction  of the corresponding system \eqref{residuo} for $\bu$ and $\bv$, that
\begin{equation}\label{esta}
-\Delta\bW-\nu\Delta\bw+\dv(\bu\otimes\bu-\bv\otimes\bv)+\nabla w=\bcero\hbox{ in }\Omega.
\end{equation}
It is the presence of the non-linear term $\dv(\bu\otimes\bu)$, so fundamental to the Navier-Stokes system, what makes the situation different compared to a linear setting. 

We write
$$
\dv(\bu\otimes\bu-\bv\otimes\bv)=\dv(\bw\otimes\bu)+\dv(\bv\otimes\bw),
$$
and bear in mind the well-know fact
\begin{equation}\label{identidades}
\int_\Omega(\bv\otimes\bv:\nabla\bu+\bu\otimes\bv:\nabla\bv)\,d\bx=\int_\Omega\bu\otimes\bv:\nabla\bu\,d\bx=0
\end{equation}
for every $\bu, \bv\in\bbD$. If we use $\bw$ as a test function in \eqref{esta}, we find that
\begin{equation}\label{uno1}
\int_\Omega[\nabla\bW:\nabla \bw+\nu|\nabla\bw|^2-(\bw\otimes\bu):\nabla\bw-(\bv\otimes\bw):\nabla\bw]\,d\bx=0.
\end{equation}
Note that the integral involving $w$ vanishes because $\bw$ is divergence-free. By \eqref{identidades}, 
\begin{gather}
\int_\Omega(\bw\otimes\bu):\nabla\bw\,d\bx=0,\nonumber\\ 
\int_\Omega(\bv\otimes\bw):\nabla\bw\,d\bx=-\int_\Omega (\bw\otimes\bw):\nabla\bv\,d\bx.\nonumber
\end{gather}
Identity \eqref{uno1} becomes
\begin{equation}\label{dos2}
\int_\Omega[\nabla\bW:\nabla \bw+\nu|\nabla\bw|^2+(\bw\otimes\bw):\nabla\bv]\,d\bx=0.
\end{equation}
We can use $\bv$ as a test function in the corresponding system \eqref{residuo} for $\bv$ to have
$$
\int_\Omega(\nabla \bV:\nabla\bv+\nu|\nabla\bv|^2-\bbf\cdot\bv)\,d\bx=0.
$$
Again we have utilized that fields in $\bbD$ are divergence-free, and the second identity in \eqref{identidades}. This last identity implies, in an elementary way, that
\begin{equation}\label{tres3}
\nu\|\bv\|_{H^1_0(\Omega; \R^N)}\le \|\bbf\|_{H^{-1}(\Omega; \R^N)}+\sqrt{2E(\bv)}.
\end{equation}
Recall that
$$
2E(\bv)=\|\bV\|^2_{H^1_0(\Omega; \R^N)}.
$$
We now have all the suitable elements to exploit \eqref{dos2}. If $C=C(n)$ is the constant of the Sobolev embedding of $H^1(\Omega)$ into $L^4(\Omega)$ for $N\le4$, then \eqref{dos2} leads to
$$
\nu\|\bw\|^2\le \|\bW\|\,\|\bw\|+C^2\|\bw\|^2\|\bv\|
$$
where all norms here are in $H^1_0(\Omega; \R^N)$. On the other hand, if we replace the size of $\bv$ by the estimate \eqref{tres3}, we are carried to 
$$
\nu\|\bw\|^2\le \|\bW\|\,\|\bw\|+\frac {C^2}\nu\|\bw\|^2\left(\|\bbf\|_{H^{-1}(\Omega; \R^N)}+\sqrt{2E(\bv)}\right),
$$
or
$$
\left(\nu-\frac {C^2}\nu\left(\|\bbf\|+\sqrt{2E(\bv)}\right)\right)\|\bw\|^2\le \|\bW\|\,\|\bw\|.
$$
Since
$$
\|\bW\|=\|\bU-\bV\|\le \|\bU\|+\|\bV\|,
$$
we would have
$$
\left(\nu-\frac {C^2}\nu\left(\|\bbf\|+\sqrt{2E(\bv)}\right)\right)^2\|\bu-\bv\|^2\le 4 (E(\bu)+E(\bv)).
$$
The form of this inequality leads us to the following interesting generalization of Definition \ref{errorgeneral}.
\begin{definition}\label{errorconcota}
A non-negative, $\C^1$-functional 
$$
E(\bu):\bbH\to\R^+
$$ 
defined over a Hilbert space $\bbH$ is called an error functional if 
there is some positive constant $c$ (including $c=+\infty$) such that:
\begin{enumerate}
\item behavior as $E'\to\bcero$:
$$
\lim_{E'(\bu)\to\bcero}E(\bu)=0
$$
over bounded subsets of $\bbH$; and
\item enhanced coercivity: there is a positive constant $C$ (that might depend on $c$), such that for every pair $\bu, \bv$ 
belonging to the sub-level set $\{E\le c\}$, 
we have
$$
\|\bu-\bv\|^2\le C(E(\bu)+E(\bv)).
$$
\end{enumerate}
\end{definition}
It is interesting to note that the sub-level sets $\{E\le d\}$ for $d<c$, for a functional $E$ verifying Definition \ref{errorconcota}, cannot maintain several connected components. 

Because our basic result Proposition \ref{basicas} is concerned with zeros of $E$, it is still valid under Definition \ref{errorconcota}.
\begin{proposition}\label{basicass}
Let $E:\bbH\to\R^+$ be an error functional according to Definition \ref{errorconcota}. Then there is a unique minimizer $\bu_\infty\in\bbH$ such that  $E(\bu_\infty)=0$, and
$$
\|\bu-\bu_\infty\|^2\le CE(\bu),
$$
for every $\bu\in\bbH$
provided $E(\bu)$ is sufficiently small ($E(\bu)\le c$, the constant in Definition \ref{errorconcota}).
\end{proposition}

The calculations that motivated this generalization yield the following.

\begin{proposition}
Let $N\le4$, and $\Omega\subset\R^N$, a bounded, Lipschitz, connected domain. If $\nu>0$ and $\bbf\in H^{-1}(\Omega; \R^N)$ are such that the quotient
$\|\bbf\|/\nu^2$ is sufficiently small, then the functional $E$ in \eqref{funcionalns} complies with Definition \ref{errorconcota}. 
\end{proposition}

We now turn to examining the interconnection between $E$ and $E'$. To this end, we gather here \eqref{residuo} and \eqref{funcionalns}
\begin{gather}
E(\bu)=\frac12\int_\Omega|\nabla\bU(\bx)|^2\,d\bx,\nonumber\\
 -\Delta\bU-\nu\Delta\bu+\dv(\bu\otimes\bu)-\bbf+\nabla u=\bcero\hbox{ in }\Omega,\nonumber
 \end{gather}
 for $\bu, \bU\in\bbD$ and $u\in L^2_0(\Omega)$. If we replace
 $$
 \bu\mapsto\bu+\epsilon\bv,\quad \bU\mapsto\bU+\epsilon\bV,\quad u\mapsto u+\epsilon v,
 $$
to first-order in $\epsilon$, we would have
\begin{gather}
E(\bu+\epsilon\bv)=\frac12\int_\Omega|\nabla\bU+\epsilon\nabla\bV|^2\,d\bx,\nonumber\\
 -\Delta(\bU+\epsilon\bV)-\nu\Delta(\bu+\epsilon\bv)+\dv((\bu+\epsilon\bv)\otimes(\bu+\epsilon\bv))\nonumber\\
 -\bbf+\nabla (u+\epsilon v)=\bcero.\nonumber
 \end{gather}
By differentiating with respect to $\epsilon$, and setting $\epsilon=0$, we arrive at
\begin{gather}
\langle E'(\bu), \bv\rangle=\int_\Omega\nabla\bU\cdot\nabla\bV\,d\bx,\nonumber\\
 -\Delta\bV-\nu\Delta\bv+\dv(\bu\otimes\bv+\bv\otimes\bu)+\nabla v=\bcero.\nonumber
 \end{gather}
If we use $\bU$ as a test function in this last system, we realize that
$$
\langle E'(\bu), \bv\rangle=\int_\Omega [-\nu\nabla\bv\cdot\nabla\bU+\bu\otimes\bv:\nabla\bU+\bv\otimes\bu:\nabla\bU]\,d\bx
$$
for every $\bv\in\bbD$. If we set $\bw=E'(\bu)\in\bbD$, then
$$
\int_\Omega[\nabla\bw\cdot\nabla \bv+\nu\nabla\bU\cdot\nabla\bv+(\nabla\bU\bu+\bu\nabla\bU)\cdot\bv]\,d\bx,
$$
for every $\bv\in\bbD$. In particular, if we plug $\bv=\bU$ in, bearing in mind that due to \eqref{identidades} the last two terms drop out, we are left with
$$
\nu\|\nabla\bU\|^2=-\langle\nabla\bw, \nabla\bU\rangle,
$$
or
$$
\nu\|\nabla\bU\|\le\|\nabla\bw\|.
$$
If the term on the right-hand side, which is $\|E'(\bu)\|$, tends to zero, so does the one on the left-hand side, which is $\nu\sqrt{2E(\bu)}$. This shows the second basic property of an error functional.
As a result, Proposition \ref{basicass} can be applied.

\begin{theorem}
If $\Omega\subset\R^N$, $N\le4$, is a bounded, Lipschitz, connected domain, and $\nu>0$ and $\bbf\in H^{-1}(\Omega; \R^N)$ in the steady Navier-Stokes system \eqref{navsto} are such that the quotient
$\|\bbf\|/\nu^2$ is sufficiently small, then there is a unique weak solution $\bu$ in $\bbD$, and
$$
\|\bu-\bv\|^2_{H^1_0(\Omega; \R^N)}\le CE(\bv)
$$
provided $E(\bv)$ is sufficiently small. 
\end{theorem}


\begin{thebibliography}{99}
\bibitem{bogunz} B. Bochev, M. Gunzburger, Least-Squares Finite Element Methods, Applied Mathematical Sciences, vol. 166, Springer, New York, 2009.
\bibitem{brezis} Brezis, Haim Functional analysis, Sobolev spaces and partial differential equations. Universitext. Springer, New York, 2011.
\bibitem{EvansB}  Evans, Lawrence C. Partial differential equations. Second edition. Graduate Studies in Mathematics, 19. American Mathematical Society, Providence, RI, 2010.
\bibitem{glowinski} Glowinski, R. (1983) Numerical Methods for Nonlinear Variational Problems, Springer series in Computational Physics, Springer, New York, NY.
\bibitem{lemmunped}  Lemoine, J., Münch, A., Pedregal, P., Analysis of continuous H1-least-squares methods for the steady Navier-Stokes system. Appl. Math. Optim. 83 (2021), no. 1, 461–488.
\bibitem{munped1} Münch, Arnaud; Pedregal, Pablo A least-squares formulation for the approximation of null controls for the Stokes system. C. R. Math. Acad. Sci. Paris 351 (2013), no. 13-14, 545–550.
\bibitem{munped2} Münch, Arnaud; Pedregal, Pablo About least-squares type approach to address direct and controllability problems. Evolution equations: long time behavior and control, 118–136, London Math. Soc. Lecture Note Ser., 439, Cambridge Univ. Press, Cambridge, 2018.
\bibitem{pedregal3} Pedregal, P., A variational approach to dynamical systems and its numerical simulation. Numer. Funct. Anal. Optim. 31 (2010), no. 7-9, 814-830.
\end{thebibliography}
\end{document}